\documentclass[12pt]{amsart}

\usepackage[margin=1.15in]{geometry}
\usepackage{amscd,amssymb, amsmath,amsfonts, wasysym, mathrsfs, mathtools,enumerate,hhline,color}
\usepackage{graphicx}
\usepackage[all, cmtip]{xy}
\usepackage{url}

\definecolor{hot}{RGB}{65,105,225}

\usepackage[pagebackref=true,colorlinks=true, linkcolor=hot ,  citecolor=hot, urlcolor=hot]{hyperref}%make all the references and links clickable

\theoremstyle{plain}
\newtheorem{theorem}{Theorem}[section]
\newtheorem{prop}[theorem]{Proposition}

\theoremstyle{definition}

\newtheorem{defn}[theorem]{Definition}

\newtheorem*{ex*}{Example}

\newtheorem*{question}{Question}

\newcommand\sO{{\mathcal O}}

\newcommand\sM{{\mathcal M}}

\newcommand\bP{{\mathbb{P}}}
\newcommand\bQ{{\mathbb{Q}}}
\newcommand\bZ{{\mathbb{Z}}}

\newcommand\bC{{\mathbb{C}}}

\newcommand\bN{{\mathbb{N}}}
\DeclareMathOperator{\Char}{Char}

\def\bC{\mathbb{C}}

\def\cD{\mathscr{D}}
\def\hh{\mathbb{H}}

\title[Algebraic surfaces with zero-dimensional cohomology support locus]{Algebraic surfaces with zero-dimensional cohomology support locus}
\begin{document}
%\author{June Huh}

\author{Botong Wang}
\email{wang@math.wisc.edu}
\address{Department of Mathematics, University of Wisconsin-Madison, Madison, WI  53706} 
\begin{abstract}
Using the theory of cohomology support locus, we give a necessary condition for the Albanese map of a smooth projective surface being a submersion. More precisely, assuming the cohomology support locus of any finite abelian cover of a smooth projective surface consists of finitely many points, we prove that the surface has trivial first Betti number, or is a ruled surface of genus one, or is an abelian surface. 
\end{abstract}

\date{}

\maketitle
\section{Introduction}
Let $X$ be a connected finite CW-complex. Denote the moduli space of rank one $\bC$-local systems on $X$ by $\Char(X)$. Then $\Char(X)$ is a commutative linear algebraic group (not necessarily connected). As algebraic groups,
$$\Char(X)\cong Hom(\pi_1(X), \bC^*)$$
which is isomorphic to the product of the affine torus $(\bC^*)^{b_1(X)}$ and a finite abelian group. Denote the connected component of $\Char(X)$ containing the origin by $\Char^0(X)$. 

The cohomology support locus of $X$ is defined to be 
$$\Sigma(X):=\{L\in \Char(X)| H^i(X, L)\neq 0, \text{ for some } i\in \bN\}.$$
The cohomology support locus $\Sigma(X)$ is always an algebraic subset of $\Char(X)$. It is a theorem of Carlos Simpson \cite{Simpson} that when $X$ is a smooth projective variety, each irreducible component of $\Sigma(X)$ is a translate of a subtorus by a torsion point. 

We are interested in the varieties with zero-dimensional cohomology support locus. Such condition gives strong consequences of the topology of $X$. For example, if $\Sigma(X)\subsetneq \Char(X)$, then the Euler characteristic $\chi(X)=0$. This is because the Euler characteristic of $X$ can be computed using any rank one local system on $X$, i.e.
$$\chi(X)=\sum_{i\geq 0}(-1)^i\dim H^i(X, L)$$
where $L$ is any rank one local system on $X$. 

Before stating our main theorem, we recall that a connected covering space $X'$ of $X$ is called an abelian cover, if the image of $p_*:\pi_1(X')\to \pi_1(X)$ is a normal subgroup and $\pi_1(X)/p_*(\pi_1(X'))$ is an abelian group. 
\begin{theorem}\label{surface}
Let $X$ be a smooth projective surface over $\bC$. Suppose that for any finite abelian cover $X'$ of $X$, $\Sigma(X')\cap \Char^0(X')$ consists of finitely many points. Then the Albanese map of $X$
$$a_X: X\to Alb(X)$$
is a submersion, or in algebraic terms, smooth and surjective. In particular, one of the following statements applies to $X$,
\begin{itemize}
\item The first Betti number of $X$ is zero. Or equivalently, the Albanese variety of $X$ is a point. 
\item $X$ is a ruled surface of genus one. 
\item $X$ is an abelian surface. 
\end{itemize}
\end{theorem}
As we will see in the next section, the local systems in $\Char^0(X)$ have better geometric interpretation. They are pull-backs from the Albanese variety of $X$. 

It is essential that we require $\Sigma(X')\cap \Char^0(X')$ consists of finitely many points for all abelian cover $X'$ of $X$. The next proposition gives a counter-example if we only require $\Sigma(X)\cap \Char^0(X)$ consists of finitely many points. The example is based on an answer on Mathoverflow by Jason Starr and communication with Christian Schnell. 
\begin{prop}\label{counter}
There exists a smooth surface $Y$ such that $\Sigma(Y)\cap \Char^0(Y)$ consists of finitely many points, but the Albanese map of $Y$ is not a submersion. 
\end{prop}
Our example of such $Y$ has torsions in $H_1(Y, \bZ)$. Hence $\Char(Y)$ has multiple connected components. In fact, one can show that $\Sigma(Y)$ contains a connected component of $\Char(Y)$, which is not $\Char^0(Y)$. We don't know whether we can replace $\Sigma(Y)\cap \Char^0(Y)$ by $\Sigma(Y)$ in the preceding proposition. Moreover, we don't know how much of our main theorem generalizes to higher dimension. So let us conclude the introduction with two questions. 
\begin{question}
Does there exist a smooth surface $Y$ such that $\Sigma(Y)$ consists of finitely many points, but the Albanese map of $Y$ is not a submersion?
\end{question}
\begin{question}
Does Theorem \ref{surface} generalize to higher dimension?
\end{question}

%In fact, for our example of $Y$, $\Char(Y)$ has multiple connected components of dimension one. And $\Sigma(Y)$ contains a component not containing the origin. So we would like to ask the following question. {\color{red}I need to double check this paragraph, and make sure our example $Y$ is not a counter example of the question. }
%\begin{question}
%Let $X$ be a smooth projective surface. Suppose $\Sigma(X)$ consists of finitely many points. Is $a_X: X\to Alb(X)$ always a submersion?
%\end{question}

\section{Cohomology support locus and Fourier-Mukai transformation}\label{pushforward}
Let $X$ be a smooth projective variety over $\bC$. The following theorem is proved implicitly by Christian Schnell in \cite{Schnell}. 
\begin{theorem}\label{localsystem}
Suppose $\Sigma(X)\cap \Char^0(X)$ consists of finitely many points. Then the Albanese map $a_X: X\to Alb(X)$ is surjective, and $Ra^i_{X*}(\bC_X)$ are local systems on $X$ for all $i\in \bN$. 
\end{theorem}
Before proving the theorem, we review some of the basic ideas of \cite{Schnell}. For simplicity of notations, we will simply denote $a_X$ and $Alb(X)$ by $a$ and $A$ respectively. 

Since the Albanese map $a$ induces an isomorphism $H_1(X, \bQ)\cong H_1(A, \bQ)$, the map $a^\star: \Char(A)\to \Char(X)$ induces an isomorphism between $\Char(A)$ and $\Char^0(X)$. Now, we define the relative cohomology support locus.

\begin{defn}
Let $F^\bullet$ be a constructible complex on $A$. We define
$$\Sigma(A, F^\bullet)=\{L\in \Char(A)|\mathbb{H}^i(A, F^\bullet\otimes_{\bC}L)\neq 0 \text{ for some } i\in \bN\}$$
where $\mathbb{H}^i$ denotes the $i$-th hypercohomology group. 
\end{defn}
%By projection formula, 
%$$H^i(X, L)\cong \mathbb{H}^i(A, Ra_*(\bC_X)).$$
%Thus, the cohomology support locus $\Sigma(X)$ (restricted to $\Char^0(X)$) is equal to the relative cohomology support locus $\Sigma(A, Ra_*(\bC_X))$. 

Next, we will review some of the arguments in \cite{Schnell}. Fix the smooth projective variety $X$. Let $A^{\natural}$ be the moduli space of rank one flat bundles on $A$. Then analytically $A^\natural$ is isomorphic to $\Char(A)$, but algebraically they are different. The isomorphism between $A^\natural$ and $\Char(A)$ is given by the Riemann-Hilbert correspondence between flat bundles and local systems.We denote this analytic isomorphism by $\Phi: A^\natural\to \Char(A)$.
\begin{defn}\label{csd}
In $A^\natural$, we can define the relative cohomology support locus of a complex of holonomic $\cD_A$-modules $\sM^\bullet$ by
\begin{equation*}
S(A, \sM^\bullet):=\{(E, \nabla)\in A^{\natural}| \hh^i(A, DR_A(A, \sM^\bullet\otimes_{\sO_A}(E, \nabla)))\neq 0 \text{ for some }i\in \bZ\}.
\end{equation*}
Here $DR_A: D^b_h(\cD_A)\to D^b_c(\bC_A)$ is the Riemann-Hilbert functor from the bounded derived category of holonomic $\cD_A$ modules to the bounded derived category of constructible complexes on $A$.
\end{defn}

The Fourier-Mukai transformation of algebraic D-modules is introduced by Laumon \cite{Laumon} and Rothstein \cite{Rothstein}. It is an equivalence of categories,
$$FM_A: D^b_{coh}(\cD_A)\to D^b_{coh}(\sO_{A^\natural})$$
between the bounded derived category of coherent $\cD_A$-modules and the bounded derived category of coherent sheaves on $A^\natural$. 

The Fourier-Mukai transformation is useful for the study of cohomology support locus, because it follows from base change theorem that 
\begin{equation}\label{equal}
Supp\left(FM_A\left(\sM^\bullet\right)\right)=S(A, \sM^\bullet)
\end{equation}
where by definition
$$Supp(FM_A(\sM^\bullet))=\bigcup_{i\in \bZ}Supp\left(H^i(FM_A(\sM^\bullet))\right).$$

\begin{proof}[Proof of Theorem \ref{localsystem}]
Notice that in Definition \ref{csd}, the cohomology support locus relative to a complex of D-modules is defined via the corresponding constructible complex. Thus,
\begin{equation}\label{iso}
\Phi(S(A, \sM^\bullet))=\Sigma(A, DR_A(\sM^\bullet)).
\end{equation}
Here we recall that $\Phi: A^\natural\to \Char(A)$ is the Riemann-Hilbert correspondence between rank one flat bundles and rank one local systems. 

Given a rank one local system $L_A$ on $A$, 
$$H^i(X, a^*(L_A))\cong \hh^i(A, Ra_*(a^*(L_A)))\cong \hh^i(A, Ra_*(\bC_X)\otimes_{\bC} L_A)$$
where the first isomorphism is because derived push-forward is associative, and the second isomorphism follows from the projection formula. Therefore,
$$\Sigma(X)\cap a^\star(\Char(A))=a^\star\left(\Sigma\left(A, Ra_*(\bC_X)\right)\right)$$
where the pull-back map $a^\star: \Char(A)\to \Char(X)$ induces an isomorphism between $\Char(A)$ and $\Char^0(X)$. 

By assumption, $\Sigma(X)\cap \Char^0(X)$ consists of finitely many points. Thus, $\Sigma(A, Ra_*(\bC_X))$ also consists of finitely many points. Since the Riemann-Hilbert correspondence is compatible with pushforward, 
\begin{equation}\label{astar}
Ra_*(\bC_X)\cong Ra_*(DR_X(\sO_X))\cong DR_A(a_*(\sO_X))
\end{equation}
where the last pushforward $a_*: D^b_h(\cD_X)\to D^b_h(\cD_A)$ is the pushforward in the derived category. 

Therefore, we have 
$$\Sigma(A, Ra_*(\bC_X))=\Sigma\left(A, DR_A(a_*(\sO_X))\right)=\Phi(S(A, a_*(\sO_X)))$$
where the last equality follows from (\ref{iso}). Since $\Phi$ is an analytic isomorphism, and since $\Sigma(A, Ra_*(\bC_X))$ consists of finitely many points, $S(A, a_*(\sO_X))$ also consists of finitely many points. By (\ref{equal}), this means that the cohomology sheaves of $FM_A(a_*(\sO_X))$ are all skyscraper sheaves. By (\ref{astar}), to show $Ra^i_{*}(\bC_X)$ are local systems on $A$ for all $i$, it suffices to show that $H^i(a_*(\sO_X))$ are smooth $\cD_A$-modules for all $i$. 

%Thus, $FM_A(a_*(\sO_X))$ splits by support. More precisely, suppose 
%$$Supp(FM_A(a_*(\sO_X)))=\{P_1, \ldots, P_l\}$$
%where $P_j$ are distinct points in $A^\natural$. Then there exists a decomposition
%$$FM_A(a_*(\sO_X))\cong\bigoplus_{1\leq j\leq l}C_i^\bullet$$
%where $C_i^\bullet\in D^b_{coh}(\sO_{A^\natural})$ is supported at the point $P_j$. 
Since $FM_A: D^b_{coh}(\cD_A)\to D^b_{coh}(\sO_{A^\natural})$ is an equivalence of category, to prove the theorem, it suffices to show that for any $C^\bullet\in D^b_{coh}(\sO_{A^\natural})$ which is supported on finitely many points in $A^\natural$, $H^i(FM_A^{-1}(C^\bullet))$ are local systems on $A$. Here $FM_A^{-1}$ is the quasi-inverse of the Fourier-Mukai functor (see \cite[Section 3]{Laumon}). 

By the definition of $FM_A^{-1}$, when $N$ is a coherent skyscraper sheaf on $A^\natural$ considered as a complex in degree zero, $FM_A^{-1}(N)$ is a flat bundle on $A$ considered as a complex in degree zero. Apply $FM_A^{-1}$ to the the sequence,
$$\cdots \tau^{\geq 1}(C^\bullet)\to \tau^{\geq 0}(C^\bullet)\to \tau^{\geq -1}(C^\bullet)\to\cdots$$
 we obtain a sequence of objects in $D^b_h(\cD_A)$,
$$\cdots\to\sM_{1}^\bullet\to \sM_{0}^\bullet\to \sM_{-1}^\bullet\to \cdots$$
where $\sM_i^\bullet=0$ when $i\gg 0$ and $\sM_i^\bullet=FM_A^{-1}(C^\bullet)$ when $i\ll 0$.

Since for any $i$ the mapping cone of $\tau^{\geq i+1}(C^\bullet)\to \tau^{\geq i}(
C^\bullet)$ is a skyscraper sheaf considered as a complex in degree $i$, the mapping cone of $M^\bullet_{i+1}\to M_i^\bullet$ is a local system considered as a complex in degree $i$. Using induction on 
$$\min\{i\;|\;\sM_i^\bullet=FM_A^{-1}(C^\bullet)\}-\max\{i\;|\;\sM_i^\bullet=0\}$$
one can easily prove that the cohomology of $FM^{-1}_A(C^\bullet)$ are local systems on $A$. Therefore, the theorem follows.
\end{proof}

\section{Proof of Theorem \ref{surface}}
We will prove Theorem \ref{surface} in this section. Let $X$ be a smooth projective surface over $\bC$ throughout the section. As in Theorem \ref{surface}, we will assume that for any finite abelian cover $X'$ of $X$, $\Sigma(X')\cap \Char^0(X')$ consists of finitely many points. 

Since the cohomology support locus $\Sigma(X)\cap \Char^0(X)$ consists of finitely many points, by Theorem \ref{localsystem}, $Ra^i_*(\bC_X)$ are local systems on $X$ for all $i$. If $Alb(X)$ has dimension at least two, $Ra^0_*(\bC_X)$ being a local system implies that the number of connected components of $a^{-1}(t)$ is constant for $t\in Alb(X)$. Thus, the Albanese map $a: X\to Alb(X)$ is surjective, and hence $\dim Alb(X)=2$. Since $\dim X=\dim Alb(X)=2$, $Ra^i_*(\bC_X)$ are supported on a proper subvariety of $Alb(X)$ for $i\geq 1$. Thus, $Ra^i_*(\bC_X)=0$ for all $i\geq 1$. In other words, $a: X\to Alb(X)$ cannot have any positive dimensional fiber. Since $X$ is projective, $a: X\to Alb(X)$ being quasi-finite implies that it is a finite morphism. Now, since $Ra^0_*(\bC_X)$ is a local system, every fiber of $a: X\to Alb(X)$ has the same number of points. Thus, $a: X\to Alb(X)$ is a finite unramified covering map. Thus, $X$ is an abelian surface. 

Now, suppose $Alb(X)$ has dimension one. Consider the Stein factorization $X\stackrel{f}{\to} Z\stackrel{g}{\to} Alb(X)$ of $a: X\to Alb(X)$. Since $f$ has connected fibers,
$$Ra^0_*(\bC_X)=R(g\circ f)^0_*(\bC_X)=Rg^0_*(Rf^0_*(\bC_X))=Rg^0_*(\bC_Z).$$
Since $Ra^0_*(\bC_X)$ is a local system on $Alb(X)$, by the argument in the previous paragraph, $g: Z\to Alb(X)$ is an unramified covering map. Thus, $Z$ is an abelian variety. By the universality of the Albanese variety, $g: Z\to Alb(X)$ is an isomorphism. Thus, $a: X\to Alb(X)$ has connected fibers. 

Next, we will show that $Ra^1_*(\bC_X)$ is trivial. Since $\pi_1(Alb(X))$ is abelian, every local systems on $Alb(X)$ is an extension of rank one local systems. According to Simpson's theorem \cite{Simpson}, $\Sigma(X)\cap \Char^0(X)$ consists of torsion points only. Thus, $Ra^1_*(\bC_X)$ is an extension of torsion local systems. Suppose $Ra^1_*(\bC_X)\neq 0$. Then there exists a finite cover $p: A'\to Alb(X)$, such that $p^*(Ra^1_*(\bC_X))$ has at least one global section. Let $X'=X\times_{Alb(X)} A'$. By choosing $p$ to be minimal, we can assume that $X'$ is connected. Then $X'$ is an abelian cover of $X$. Moreover, by projection formula, $Ra'^1_*(\bC_{X'})\cong p^*(Ra^1_*(\bC_X))$, where $a': X' \to A'$ is the projection. Now, by the degeneration of the Leray spectral sequence \cite{Deligne}, 
$$\dim_{\bC} H^1(X', \bC_{X'})=\dim_{\bC} H^1(A', Ra'^0_*(\bC_{X'}))+\dim_{\bC} H^0(A', Ra'^1_*(\bC_{X'})).$$
Since $p^*(Ra^1_*(\bC_X))$ has at least one global section, and since $Ra'^1_*(\bC_{X'})\cong p^*(Ra^1_*(\bC_X))$,
$$\dim_{\bC} H^0(A', Ra'^1_*(\bC_{X'}))=\dim_{\bC} H^0(A', p^*(Ra^1_*(\bC_X)))\geq 1.$$
Since $a: X\to Alb(X)$ has connected fibers, so does $a': X'\to A'$, and hence
$$\dim_{\bC} H^1(A', Ra'^0_*(\bC_{X'}))=\dim_\bC H^1(A', \bC_{A'})=2.$$
Therefore, $\dim_\bC H^1(X', \bC_{X'})\geq 3$, and hence $Alb(X')$ has dimension at least two. By the previous arguments, $\Sigma(X')\cap \Char(X')$ consisting of finitely many points implies that $X'$ is an abelian surface. Thus, $X$ is an abelian surface too. This is a contradiction to the assumption that $Alb(X)$ has dimension one. Hence, $Ra^1_*(\bC_X)=0$. 

Thus, we have shown that the general fiber of $a: X\to Alb(X)$ is isomorphic to $\bP^1$. Next, we show that $X$ is a minimal surface. In fact, suppose $X$ is not minimal, i.e., there exists a blow-down map $f: X\to X_0$. Then the Albanese map $a: X\to Alb(X)$ factors as $X\stackrel{f}{\to} X_0\stackrel{g}{\to} Alb(X)$. Since $f$ is a blow-down map, $Rf^2_*(\bC_X)$ has a skyscraper sheaf as a direct summand. In fact, by the decomposition theorem \cite{BBD}, $Rf_*(\bC_X)$ has a direct summand, which is supported at a point. The image of this direct summand under $Rg_*$ will be a direct summand of $Rg_*(Rf_*(\bC_X))\cong Ra_*(\bC_X)$. This contradicts that $Ra^i_*(\bC_X)$ are local systems for all $i$. So $X$ is a minimal surface. 

By the adjunction formula, the canonical bundle of $X$ is not nef. By \cite[VI, Proposition 3.3]{BHPV}, $X$ is either $\bP^2$ or a ruled surface. Since $Alb(X)$ has dimension one, $X$ is a ruled surface of genus one. 

The last case is when $Alb(X)$ has dimension zero. The statement of Theorem \ref{surface} is vacant in this case. 
\section{Constructing examples for Proposition \ref{counter}}
We will construct an example to show Proposition \ref{counter}. 

Let $\pi: C\to E$ be a double cover of an elliptic curve $E$ ramified at two points $P$ and $Q$. Let $\sigma: C\to C$ be the Galois map relative to $\pi: C\to E$. Let $V_\bZ$ be a rank one $\bZ$ coefficient local system on $E$ such that $V_\bZ$ is nontrivial, but $V_\bZ\otimes_\bZ V_\bZ$ is trivial. Denote the $\bC$ coefficient local system $V_\bZ\otimes_\bZ \bC$ by $V_\bC$. Equivalently, we can define $V_\bC$ as a non-trivial rank one $\bC$-local system on $E$, whose monodromy group is $\{1, -1\}$, and we define $V_\bZ$ as the $\bZ$-sections of $V_\bC$.   Denote $\pi^*(V_\bZ)$ and $\pi^*(V_\bC)$ by $W_\bZ$ and $W_\bC$ respectively. Clearly, $W_\bC/(W_\bZ\otimes_\bZ \bZ[\sqrt{-1}])$ is a principal $\bC/\bZ[\sqrt{-1}]$-bundle over $C$ with a zero section. Here $\bZ[\sqrt{-1}]=\bZ\oplus \bZ\cdot\sqrt{-1}$ is the subring generated by $\sqrt{-1}$ over $\bZ$. Denote the total space of this principal bundle by $\tilde{Y}$. Then $\tilde{Y}$ is an isotrivial family of elliptic curves over $C$ with a zero section. 

On $\tilde{Y}$, there are a few order two actions. The first one is induced by the Galois action $\sigma$, which by abusing notations we also denote by $\sigma$. Notice that ``$\pm \frac{1}{2}$" is a double section of $V_\bC$, where the two sections differ by a translation by $V_\bZ$. Thus, it descents to a global section of $\tilde{Y}\to C$, which we denote by $s$. Then addition by $s$ along the fibers defines a fixed point free action on $\tilde{Y}$ of order two. We denote the second action by $\tau$. Evidently, the two actions $\sigma$ and $\tau$ commute with each other. Thus, $\sigma\circ \tau$ is also an action of order two on $\tilde{Y}$, which is obviously fixed point free. We define $Y$ to be the quotient of $\tilde{Y}$ by the action $\sigma\circ\tau$. 

Since $\sigma\circ\tau$ preserves each fiber of $\tilde{Y}\to E$, we have a map $\phi: Y\to E$. It is easy to check that away from $P$ and $Q$, $\phi: Y\to E$ is smooth with fiber isomorphic to the elliptic curve $\bC/\bZ[\sqrt{-1}]$. However, over $P$ and $Q$, $\phi$ has non-reduced fiber. The fibers over $P$ and $Q$ have multiplicity two and are isomorphic to $\bC/(\bZ[\sqrt{-1}]+\frac{1}{2}\bZ[\sqrt{-1}])$. 

Notice that the map $\phi: Y\to E$ defined above is the Albanese map of $Y$. In fact, let $F$ be a general fiber of $\phi$. One can easily check that the map $H_1(F, \bQ)\to H_1(Y, \bQ)$ induced by the inclusion is zero. Therefore, the Albanese map contracts the fiber $F$ to a point. Thus, $\phi: Y\to E$ is the Albanese map of $Y$. It follows from straightforward computation that $R\phi^i_*(\bC_Y)$ are local systems on $E$ for all $i$. By the decomposition theorem \cite{BBD}, this implies that 
$$R\phi_*(\bC_Y)\cong \bigoplus_{i\geq 0} R\phi^i_*(\bC_Y)[i].$$
Notice that any local system on an elliptic curve is obtained by extensions of rank one local systems. Therefore, by the arguments in Section \ref{pushforward}, $\Sigma(Y)\cap \Char^0(Y)=\phi^\star \Sigma(E, R\phi_*(\bC_Y))$ consists of finitely many points. 

\subsection*{Acknowledgements}
We thank Sai-Kee Yeung for valuable discussions. The main theorem is in response to one of his questions. We also thank Donu Arapura, Christian Schnell, Partha Solapurkar and Jason Starr for helpful communications.

\end{document}